\documentclass[a4paper]{amsart}

\usepackage[utf8]{inputenc}
\usepackage{enumitem}
\usepackage{setspace}
\usepackage{mathtools}
\usepackage{tikz-cd}
\usepackage{lipsum}
\usepackage{tikz}
\usepackage{mathrsfs}
\usetikzlibrary{arrows}
\usepackage{amssymb}
\usetikzlibrary{decorations.pathmorphing}
\usepackage{bbold}
\usepackage{amsmath,amsthm}
\usepackage{xspace}
\usepackage{faktor}
\usepackage{calrsfs}
\usepackage{verbatim}
\usepackage{quiver}
\usepackage[pdfa]{hyperref}
\usepackage{amsfonts}
\usepackage{graphicx}
\usepackage{eucal}
\usepackage{amscd}
\usepackage[all,2cell]{xy}
\UseAllTwocells

\newtheorem{teo}{Theorem}[section]

\newtheorem{lem}[teo]{Lemma}
\newtheorem{defi}[teo]{Definition}
\newtheorem{oss}[teo]{Remark}
\newtheorem{prop}[teo]{Proposition}

\newcommand{\op}[1]{#1^{\mathrm{op}}}

\DeclareMathOperator{\Zker}{\mathcal{Z}-ker}

\newdir{ m>}{{}*!/-3.5pt/@{}*!/-8pt/:(1,-.2)@^{>}*!/-8pt/:(1,+.2)@_{>}}

\def\pullback{
 \ar@{-}[]+R+<6pt,-1pt>;[]+RD+<6pt,-6pt>%
 \ar@{-}[]+D+<1pt,-6pt>;[]+RD+<6pt,-6pt>}

\begin{document}

\title[Homological lemmas in a non-pointed context]{Homological lemmas in a non-pointed context}

\author[A. Cappelletti]{Andrea Cappelletti}
\address[Andrea Cappelletti]{Dipartimento di Matematica, Universit\`{a} degli Studi di Salerno, Via Giovanni Paolo II 132, 84084 Fisciano (SA), Italy}
\thanks{}
\email{acappelletti@unisa.it}

\author[A. Montoli]{Andrea Montoli}
\address[Andrea Montoli]{Dipartimento di Matematica ``Federigo Enriques'', Universit\`{a} degli Studi di Milano, Via Saldini 50, 20133 Milano, Italy}
\thanks{} \email{andrea.montoli@unimi.it}



\begin{abstract}
We show that non-pointed versions of the classical homological lemmas hold in regular protomodular categories equipped with a suitable posetal monocoreflective subcategory. Examples of such categories are all protomodular varieties of universal algebras having more than one constant, like the ones of unitary rings, Boolean algebras, Heyting algebras and MV-algebras, their topological models, and the dual category of every elementary topos.
\end{abstract}

\date{\today}

\maketitle

\section{Introduction}
Homological categories, namely pointed regular protomodular categories, have been shown to constitute the good context in which the non-abelian versions of the classical homological lemmas, such as the short five lemma, the nine lemma, and the Noether isomorphism theorems, hold (see \cite{Bbbook} for a detailed description of homological categories and their properties). In particular, for pointed categories, protomodularity \cite{Bourn protomod} is equivalent to the validity of the split short five lemma and, for pointed regular categories, protomodularity is equivalent to the validity of the short five lemma. A key fact in order to have a good behaviour of (short) exact sequences is that, in a pointed protomodular category, every regular epimorphism is a cokernel. Examples of homological categories are those of groups, non-unitary rings, associative algebras, Lie algebras, topological groups and many others.

Concerning the nine lemma, Bourn proved in \cite{Bourn 3x3} that it holds in every homological category, and also in every regular protomodular quasi-pointed category (meaning that the unique morphism from the initial object to the terminal one is a monomorphism). Replacing short exact sequences with exact forks, he considered a denormalized version of the nine lemma, proving in \cite{Bourn denormalized 3x3} that it holds in regular Mal'tsev \cite{CLP Malt'sev cats} categories. Later, Lack \cite{Lack 3x3} showed that the denormalized nine lemma holds also in regular Goursat \cite{CLP Malt'sev cats, CKP Goursat cats} categories. In \cite{GJR 3x3 lemma} the authors proposed a framework for a common description of these two versions of the nine lemma. Such a unified framework is the one of star-regular categories \cite{GJU star-regular}, which is based on the notion of an ideal of morphisms in the sense of \cite{Ehresmann}. Given an ideal $\mathcal{N}$ of morphisms in a category $\mathcal{C}$, a star is a pair $(k_1, k_2)$ of parallel morphisms such that $k_1 \in \mathcal{N}$. The \emph{star-kernel} of a morphism $f$ is a universal star with respect to the property that $fk_1 = fk_2$. A regular category equipped with an ideal $\mathcal{N}$ of morphisms is \emph{star-regular} if every regular epimorphism is the coequalizer of its star-kernel. In \cite{GJR 3x3 lemma} it is shown that, in a star-regular category with ``enough trivial objects'' the upper and the lower nine lemmas are equivalent. Moreover, under mild additional assumptions, the middle version of the nine lemma is equivalent to a version of the short five lemma relative to stars. These results cover the known ones concerning the homological lemmas in the pointed and quasi-pointed contexts (where $\mathcal{N}$ is the class of morphisms that factor through $0$) as well as the denormalized versions of them (where $\mathcal{N}$ is the class of all morphisms).

However, this context excludes several interesting examples in which some forms of the nine lemma are valid, like unitary rings, Boolean algebras, Heyting algebras, MV-algebras and, more generally, protomodular varieties of universal algebras having more than one constant. The aim of the present paper is to introduce a categorical framework which includes all the examples just mentioned, and in which suitable forms of the homological lemmas hold. We consider regular protomodular categories $\mathcal{C}$ equipped with a full, posetal, monocoreflective subcategory $\mathcal{Z}$ of ``zero objects'' such that the reflector inverts monomorphisms.

Our context is related to the one considered in \cite{GrandisJanMarki} in order to study radical theory and closure operators in a non-pointed situation. However, in \cite{GrandisJanMarki} the authors considered subcategories that are both reflective and coreflective, and so their framework does not include the main examples we are interested in. Pointed and quasi-pointed regular protomodular categories are examples of our situation. Another large class of examples is given by regular protomodular categories with initial object in which the unique morphism $0 \to 1$ is a regular epimorphism. This includes, in particular, \emph{ideally exact categories} in the sense of \cite{Janelidze ideally exact} (that are, moreover, required to be Barr-exact), among which there are the dual categories of all elementary toposes and all protomodular varieties of universal algebras with more than one constant. Moreover, the topological models of protomodular theories with more than one constant are examples of our situation (thanks to the observation, made in \cite{BorceuxClementino protomod top alg}, that such models are regular protomodular categories).

An apparently disappointing fact happening in our context is that, denoting with $\mathcal{N}_{\mathcal{Z}}$ the ideal of morphisms in $\mathcal{C}$ that factor through $\mathcal{Z}$, $\mathcal{N}_{\mathcal{Z}}$-kernels always exist and are easy to compute, while $\mathcal{N}_{\mathcal{Z}}$-cokernels do not always exist and, moreover, regular epimorphisms are not always $\mathcal{N}_{\mathcal{Z}}$-cokernels. This creates a difference, with respect to the classical case and to the one of star-regular categories, concerning the definition of short exact sequence. We show that this difficulty can be encompassed by defining a short exact sequence as a regular epimorphism with its $\mathcal{N}_{\mathcal{Z}}$-kernel. Using this notion of short exact sequence, we recover, in our wide context, the validity of the short five lemma and of the nine lemma.

\section{The context}
We work in a category $\mathcal{C}$ with a fixed full, posetal, monocoreflective subcategory $\mathcal{Z}$. We denote by $\mathsf{Z} \colon \mathcal{C} \to \mathcal{Z}$ the coreflector (and, without loss of generality, we suppose that $\mathsf{Z}$ is the identity on the subcategory $\mathcal{Z}$). Moreover, we require that $\mathsf{Z}$ inverts monomorphisms, meaning that, if $m$ is a monomorphism, then $\mathsf{Z}(m)$ is an isomorphism. We will call \emph{trivial} or \emph{zeros} the objects in the class $\mathcal{Z}$.

\begin{prop}\label{contesto generale}
Consider a category $\mathcal{C}$ and a full subcategory $\mathcal{Z}$. The following conditions are equivalent:
\begin{itemize}
    \item[(a)] $\mathcal{Z}$ is a posetal monocoreflective subcategory such that the coreflector inverts monomorphisms;
    \item[(b)] \begin{itemize}
        \item[(i)] for every $A \in \mathcal{C}$ there exists a, unique up to isomorphisms, monomorphism $\varepsilon_A \colon Z \rightarrowtail A$, with $Z \in \mathcal{Z}$;
        \item[(ii)] for every $A \in \mathcal{C}$ and $Z \in \mathcal{Z}$, $\text{Hom}(Z,A)$ has at most one element;
        \item[(iii)] for every pair of arrows $z \colon Z \rightarrowtail A$ and $z' \colon Z' \to A$, where $z$ is a monomorphism and $Z,Z' \in \mathcal{Z}$, there exists a unique morphism $\varphi \colon Z' \to Z$ such that $z\varphi=z'$:
\[ \xymatrix{ Z' \ar[d]_{\exists ! \varphi} \ar[r]^{z'} & A \\
Z. \ar@{ m>->}[ur]_z & } \]
    \end{itemize}
\end{itemize}
\end{prop}

\begin{proof}
(b)$\Rightarrow$(a) For every $A \in \mathcal{C}$ we fix an object of $\mathcal{Z}$, denoted by $\mathsf{Z}(A)$, such that there is a monomorphism $\varepsilon_A \colon \mathsf{Z}(A) \rightarrowtail A$ (such an object is unique up to isomorphisms). We want to define a functor $\mathsf{Z} \colon \mathcal{C} \to \mathcal{Z}$. Given an arrow $f \colon A \to B$ of $\mathcal{C}$, we define $\mathsf{Z}(f)$ thanks to Condition (iii):
\[ \xymatrix{ \mathsf{Z}(A) \ar[d]_{\exists ! \mathsf{Z}(f)} \ar@{ m>->}[r]^{\varepsilon_A} & A \ar[r]^f & B \\
\mathsf{Z}(B). \ar@{ m>->}[urr]_{\varepsilon_B} & & } \]
Clearly $\mathsf{Z}$ is a functor. We prove that $\mathsf{Z}$ is right adjoint to the inclusion $i \colon \mathcal{Z} \hookrightarrow \mathcal{C}$. We do this by showing that the collection of arrows given by $\varepsilon_A$ for $A \in \mathcal{C}$ satisfies the universal property of the counit. In fact, for every arrow $z \colon Z \to A$ (where $Z \in \mathcal{Z}$), thanks to Condition (iii), there exists a unique $\varphi \colon Z \to \mathsf{Z}(A)$ such that the following diagram is commutative:
\[ \xymatrix{ & Z \ar[dl]_{\exists ! \varphi} \ar[dr]^z & \\
\mathsf{Z}(A) \ar@{ m>->}[rr]_{\varepsilon_A} & & A. } \]
Hence $\mathcal{Z}$ is a monocoreflective subcategory of $\mathcal{C}$. Now, thanks to Condition (ii) we know that $\mathcal{Z}$ is a posetal category. Finally, we need to show that $\mathsf{Z}$ inverts monomorphisms. Consider a monomorphism $m \colon A \rightarrowtail B$; since $\varepsilon_B \mathsf{Z}(m)=m \varepsilon_A$ we deduce that $\mathsf{Z}(m)$ is a monomorphism of $\mathcal{C}$, too. Now, thanks to (iii) we conclude that $\mathsf{Z}(m)$ is an isomorphism, considering the commutative diagram
\[ \xymatrix{ \mathsf{Z}(B) \ar[d]_{\exists ! \theta} \ar@{=}[r] & \mathsf{Z}(B) \\
\mathsf{Z}(A). \ar@{ m>->}[ur]_{\mathsf{Z}(m)} } \]

(a)$\Rightarrow$(b)
\begin{itemize}
\item[(i)] Since $\mathsf{Z}$ is a monocoreflector, we can take as monomorphism \linebreak $\varepsilon_A \colon \mathsf{Z}(A) \rightarrowtail A$; this monomorphism is unique up to isomorphisms since $\mathsf{Z}$ inverts monomorphisms.
\item[(ii)] Consider a pair of arrows $f_1,f_2 \colon Z \rightarrow A$. Thanks to the universal property of the counit, we observe that there exist $\varphi_1, \varphi_2 \colon Z \rightarrow \mathsf{Z}(A)$ such that $\varepsilon_A \varphi_1=f_1$ and $\varepsilon_A \varphi_2=f_2$. But $\mathcal{Z}$ is a posetal category, hence $\varphi_1=\varphi_2$ and so $f_1=f_2$.
\item[(iii)] Consider a pair of arrows $z \colon Z \rightarrowtail A$ and $z' \colon Z' \to A$, where $z$ is a monomorphism and $Z,Z' \in \mathcal{Z}$. Due to naturality, we observe that $\varepsilon_A \mathsf{Z}(z)=z$. Finally, thanks to the universal property of the counit $\varepsilon$, we deduce that there exists a unique $\psi$ such that $\varepsilon_A \psi = z'$; hence, we define $\varphi= \mathsf{Z}(z)^{-1} \psi$ and we observe that $z \varphi=z \mathsf{Z}(z)^{-1} \psi= \varepsilon_A \psi= z'$.
\end{itemize}
\end{proof}

\begin{oss} \label{uniqueness of trivial class}
In a category $\mathcal{C}$ there is at most one subcategory $\mathcal{Z}$ satisfying the conditions of Proposition \ref{contesto generale}. So, admitting such a subcategory is a property of the category $\mathcal{C}$.
\end{oss}

\begin{proof}
Suppose $\mathcal{C}$ has two full, replete subcategories $\mathcal{Z}$ and $\mathcal{Y}$ satisfying the conditions of Proposition \ref{contesto generale}, with coreflectors $\mathsf{Z}$ and $\mathsf{Y}$, respectively. Let $Y$ be an object of $\mathcal{Y}$; then we have a monomorphism $\varepsilon_Y \colon \mathsf{Z}(Y) \to Y$. Its image $\mathsf{Y}(\varepsilon_Y) \colon \mathsf{Y}(\mathsf{Z}(Y)) \to \mathsf{Y}(Y)=Y$ under the coreflector $\mathsf{Y}$ is an isomorphism. Since $\mathsf{Y}(\mathsf{Z}(Y))$ is a subobject of $\mathsf{Z}(Y)$, we get that $Y \cong \mathsf{Z}(Y)$ belongs to $\mathcal{Z}$, so that $\mathcal{Y} \subseteq \mathcal{Z}$. The other inclusion can be proved similarly.
\end{proof}

Given an object $A$ of $\mathcal{C}$, we will say that $\mathsf{Z}(A)$ is the \emph{zero part} of $A$.

\begin{prop}
Under the equivalent conditions of Proposition \ref{contesto generale}, we have:
\begin{itemize}
    \item[(a)] every morphism $f \colon A \to Z$, where $Z \in \mathcal{Z}$, is a strong epimorphism;
    \item[(b)] every monomorphism $s \colon S \rightarrowtail Z$, where $Z \in \mathcal{Z}$, is an isomorphism.
\end{itemize}
\end{prop}

\begin{proof}
\begin{itemize}
    \item[(a)] Consider the commutative square
\[ \xymatrix{ A \ar[d]_g \ar[r]^f & Z \ar[d]^h \\
B \ar@{ m>->}[r]_m & C, } \]
where $m$ is a monomorphism. We recall that $\mathsf{Z}(m) \colon \mathsf{Z}(B) \to \mathsf{Z}(C)$ is an isomorphism. Moreover, since $Z \in \mathcal{Z}$, there exists a morphism $\theta \colon Z \to \mathsf{Z}(C)$ such that $\varepsilon_C \theta = h$. Hence, we can define $d = \varepsilon_B \mathsf{Z}(m)^{-1} \theta$ and we observe that $m d= m \varepsilon_B \mathsf{Z}(m)^{-1} \theta= \varepsilon_C \mathsf{Z}(m) \mathsf{Z}(m)^{-1} \theta= \varepsilon_C \theta = h$; moreover, since $m$ is a monomorphism, we get $df=g$. Therefore, $f$ is a strong epimorphism.

\item[(b)] follows immediately from (a).
\end{itemize}
\end{proof}

We denote by $\mathcal{N}_{\mathcal{Z}}$ the class of arrows of $\mathcal{C}$ factorizing
through an object of $\mathcal{Z}$. $\mathcal{N}_{\mathcal{Z}}$ is clearly an ideal of morphisms in the sense of \cite{Ehresmann}. Moreover:

\begin{oss}
If $f \in \mathcal{N}_{\mathcal{Z}}$, then we can construct a factorization $f=m \chi$ of $f$ through an object of $\mathcal{Z}$, with $m$ a monomorphism. In fact, since $f \in \mathcal{N}_{\mathcal{Z}}$, we have a factorization of the form
\[ \xymatrix{ A \ar[dr]_a \ar[rr]^f & & B \\
& Z, \ar[ur]_b & } \]
where $Z \in \mathcal{Z}$. Observing that $b=\varepsilon_B \mathsf{Z}(b)$, we obtain $f=\varepsilon_B \mathsf{Z}(b)a$, and $\varepsilon_B$ is a monomorphism.
\end{oss}

\begin{prop} \label{cancellation in NZ with strong epis}
If $f$ is a strong epimorphism and $gf \in \mathcal{N}_{\mathcal{Z}}$, then $g \in \mathcal{N}_{\mathcal{Z}}$.
\end{prop}

\begin{proof}
Since $gf \in \mathcal{N}_{\mathcal{Z}}$, we get $gf=\varepsilon_C \chi$ for some arrow $\chi$, i.e.\ the diagram below is commutative:
\[ \xymatrix{ A \ar[d]_{\chi} \ar[r]^f & B \ar[d]^g \\
\mathsf{Z}(C) \ar@{ m>->}[r]_{\varepsilon_C} & C. } \]
Therefore, recalling that $f$ is a strong epimorphism, there exists a morphism $d \colon B \to \mathsf{Z}(C)$ such that $\varepsilon_C d=g$, and so $g \in \mathcal{N}_{\mathcal{Z}}$.
\end{proof}

\section{Examples}
In this section we describe several examples of the situation considered in the previous one. We first consider a general class of categories satisfying the conditions of Proposition \ref{contesto generale}, and then some concrete cases.

Let $\mathcal{C}$ be a regular category with initial object $0$. We denote by $\mathcal{Z}$ the class of objects of $\mathcal{C}$ that are regular quotients of $0$, i.e.\ those objects $Z$ such that the unique morphism $0 \to Z$ is a regular epimorphism.

\begin{prop}
The class $\mathcal{Z}$ defined above satisfies the equivalent conditions of Proposition \ref{contesto generale}.
\end{prop}

\begin{proof}
We prove that $\mathcal{Z}$ satisfies the conditions stated in (b).
\begin{itemize}
\item[(i)] We define $\varepsilon_A \colon \mathsf{Z}(A) \rightarrowtail A$ by taking the monomorphic part of the regular epi-mono factorization of $i_A \colon 0 \to A$. Suppose $m \colon Z \rightarrowtail A$ and $m' \colon Z' \rightarrowtail A$ are subobjects of $A$, with $Z, Z' \in \mathcal{Z}$. Then we have the following commutative diagram:
\[ \xymatrix{ & Z \ar@{ m>->}[dr]^m & \\
0 \ar@{->>}[ur] \ar@{->>}[dr] & & A \\
& Z'. \ar@{ m>->}[ur]_{m'} & } \]
Thanks to the uniqueness, up to isomorphisms, of the regular epi-mono factorization, we get that $Z$ and $Z'$ are isomorphic.

\item[(ii)] Given two morphisms $f, g \colon Z \to A$, composing them with the unique morphism $i_Z \colon 0 \to Z$ we get $f i_Z = g i_Z$, and then $f = g$ because $i_Z$ is an epimorphism.

\item[(iii)] Given an arrow $z' \colon Z' \to A$ we take its regular epi-mono factorization $z'=me$ and we consider the commutative diagram
\[ \xymatrix{ 0 \ar[rr]^{i_A} \ar@{->>}[dr]_{i_{Z'}} & & A \\
& Z' \ar[ur]^{z'} \ar@{->>}[r]_e & Z, \ar@{ m>->}[u]_m } \]
where $i_{Z'}$ is a regular epimorphism thanks to the definition of the class $\mathcal{Z}$, and $Z \in \mathcal{Z}$ since it is a quotient of $Z'$. Due to the uniqueness of the regular epi-mono factorization, we obtain that $m$ is isomorphic to $\varepsilon_A$, and so the claim follows.
\end{itemize}
\end{proof}

Thanks to Remark \ref{uniqueness of trivial class}, we can observe that, for a regular category with an initial object $0$, the class of regular quotients of $0$ is the only subcategory satisfying the conditions of Proposition \ref{contesto generale}. Moreover, if a category $\mathcal{C}$ with products is equipped with a subcategory $\mathcal{Z}$ as in Proposition \ref{contesto generale} which is essentially small, then the initial object of $\mathcal{C}$ can be constructed using $\mathcal{Z}$ (this was pointed out to us by the referees):

\begin{prop}
Let $\mathcal C$ be a category with products, and $\mathcal Z$ a subcategory that meets the conditions of Proposition \ref{contesto generale}. If $\mathcal Z$ is essentially small, then $\mathcal C$ has an initial object.
\end{prop}

\begin{proof}
Since $\mathcal Z$ is essentially small, there is a set $\{Z_i \in \mathcal Z \,|\, i\in I \}$ such that for every $Z \in \mathcal Z$ there exists an element $i \in I$ and an isomorphism $\varphi \colon Z_i \to Z$. We define $A \coloneqq \prod_{i \in I} Z_i$ (the product is computed in $\mathcal C$) and we prove that $\mathsf{Z}(A)$ is an initial object of $\mathcal C$. To show this, it suffices to prove that, for every $X \in \mathcal C$, there exists an arrow $\mathsf{Z}(A) \to X$ (the uniqueness is guaranteed by Proposition \ref{contesto generale}). Given such an object $X$, we have an isomorphism $\varphi \colon Z_i \to \mathsf{Z}(X)$, for an appropriate $i \in I$. Hence, we can consider the composite
\[ \xymatrix{ \mathsf{Z}(A) \ar[r]^-{\mathsf{Z}(\pi_i)} & Z_i \ar[r]^-{\varphi} & \mathsf{Z}(X) \ar[r]^-{\varepsilon_X} & X. } \]
\end{proof}

Regular pointed and quasi-pointed categories are examples of our situation. In these cases, the only quotient of $0$, up to isomorphisms, is $0$ itself, so the ideal $\mathcal{N}_{\mathcal{Z}}$ is the class of morphisms that factor through $0$. \\

We will be particularly interested in the case of regular categories in which the unique morphism $0 \to 1$ is a regular epimorphism and, again, $\mathcal{Z}$ is the subcategory whose objects are the regular quotients of $0$. In particular, every variety of universal algebras with at least one constant is an example of such category. Important non-pointed cases are:
\begin{itemize}
    \item the categories $\mathbf{Boole}$ of Boolean algebras, $\mathbf{Heyt}$ of Heyting algebras and $\mathbf{MV}$ of MV-algebras: in these cases the initial object is the $2$-element algebra, so the class $\mathcal{Z}$ is given just by $0$ and $1$;
    \item the category $\mathbf{Ring}$ of unitary rings: in this case the class $\mathcal{Z}$ is given by the collection of the quotients of $\mathbb{Z}$. Two unitary rings have the same zero part if and only if they have the same characteristic.
\end{itemize}

The concrete examples mentioned above are protomodular categories. Other examples of regular categories in which the morphism $0 \to 1$ is a regular epimorphism are the topological models of protomodular algebraic theories with at least one constant (e.g.\ topological unitary rings). Indeed, as shown in \cite{BorceuxClementino protomod top alg}, these categories are regular and protomodular; moreover, since the terminal object is still the singleton, the morphism $0 \to 1$ is a regular epimorphism.

Another interesting class of examples is given by the categories of the form $\op{\mathcal{E}}$, where $\mathcal{E}$ is an elementary topos; in this case the class $\mathcal{Z}$ is $\{ Z \in \op{\mathcal{E}} \, | \, \exists \, Z \rightarrowtail 1 \text{ in } \mathcal{E} \}$. In this situation, two objects have the same zero part if and only if they determine the same truth value in $\mathcal{E}$ (where an object $A$ in $\mathcal{E}$ determines the truth value given by the monomorphic part of the regular epi-mono factorization of $A \to 1$).

\section{$\mathcal{Z}$-kernels and $\mathcal{Z}$-cokernels}
Let $\mathcal{C}$ be a regular category equipped with a full subcategory $\mathcal{Z}$ satisfying the equivalent conditions of Proposition \ref{contesto generale}. Since the class $\mathcal{N}_{\mathcal{Z}}$ of morphisms that factor through $\mathcal{Z}$ is an ideal, we can consider kernels and cokernels with respect to it, as in \cite{Ehresmann2}:

\begin{defi}
Let $f \colon A \to B$ be a morphism in $\mathcal{C}$. We say that a morphism $k \colon K \to A$ in $\mathcal{C}$ is
a \emph{$\mathcal{Z}$-kernel} of $f$ if the following properties are satisfied:
\begin{itemize}
    \item[(a)] $fk \in \mathcal{N}_{\mathcal{Z}}$;
    \item[(b)] whenever $e: E \to A$ is a morphism in $\mathcal{C}$ and $fe \in \mathcal{N}_{\mathcal{Z}}$, then there exists a unique morphism $\varphi \colon E \to K$ in $\mathcal{C}$ such that $k \varphi = e$.
\end{itemize}
\end{defi}
The definition of a $\mathcal{Z}$-cokernel is dual. \\

It is immediate to see (see e.g.\ \cite{Grandis}) that every $\mathcal{Z}$-kernel is a monomorphism and that the $\mathcal{Z}$-kernel of any morphism $f \colon A \rightarrow B$ in $\mathcal{C}$, if it exists, is unique up to a unique isomorphism. This means that if $k \colon K \rightarrow A$ and $k' \colon K' \rightarrow A$ are $\mathcal{Z}$-kernels of the same arrow $f$, then there exists a unique isomorphism $\varphi \colon K' \rightarrow K$ such that $k \varphi = k'$. The duals of the previous observations hold for $\mathcal{Z}$-cokernels. \\

Actually, in our context $\mathcal{Z}$-kernels always exist, and they are obtained as pullbacks along the zero part of the codomain (this result should be compared with Proposition $2.3$ in \cite{GrandisJanMarki}, where the same fact is stated in a more restrictive context):

\begin{prop}
Under our assumptions, for every arrow $f \colon A \to B$ in $\mathcal{C}$ the $\mathcal{Z}$-kernel $k \colon K \to A$ of $f$ exists and it is given by the pullback
\[ \xymatrix{ K \ar[r]^{\chi} \ar[d]_k \pullback & \mathsf{Z}(B) \ar@{ m>->}[d]^{\varepsilon_B} \\
A \ar[r]_f & B. } \]
\end{prop}

\begin{proof}
Clearly $fk \in \mathcal{N}_{\mathcal{Z}}$. Consider a morphism $l \colon L \to A$ such that $fl \in \mathcal{N}_{\mathcal{Z}}$. Then, there exists a morphism $\chi'$ making the following diagram commutative:
\[ \xymatrix{ L \ar[r]^l \ar[dr]_{\chi'} & A \ar[r]^f & B \\
& \mathsf{Z}(B). \ar@{ m>->}[ur]_{\varepsilon_B} } \]
Since $fl=\varepsilon_B \chi'$, thanks to the universal property of pullbacks there exists a unique morphism $\varphi$ such that $k \varphi=l$ and $\chi \varphi=\chi'$. Finally, consider a morphism $\theta$ such that $k \theta = l$; observing that $k$ is a monomorphism ($k$ is the pullback of the monomorphism $\varepsilon_B$), we conclude $\varphi = \theta$.
\end{proof}

The situation concerning cokernels is different. Indeed, in our context, $\mathcal{Z}$-cokernels do not always exist. To illustrate this, let us consider the category $\mathbf{Ring}$ of unitary rings. The identity morphism $1_{\mathbb{Z} \times \mathbb{Z}}$ does not have a $\mathcal{Z}$-cokernel. Indeed, let us suppose the existence of a $\mathcal{Z}$-cokernel $q \colon \mathbb{Z} \times \mathbb{Z} \rightarrow Q$ of $1_{\mathbb{Z} \times \mathbb{Z}}$. Therefore, considering the diagram below, there would be two morphisms $\varphi_1, \varphi_2 \colon Q \rightarrow \mathbb{Z}$ such that $\varphi_1 q= \pi_1$ and $\varphi_2 q = \pi_2$, where $\pi_1$ and $\pi_2$ are the product projections:
\[ \xymatrix{ & & \mathbb{Z} \\
\mathbb{Z} \times \mathbb{Z} \ar@/^1pc/[urr]^{\pi_1} \ar@/_1pc/[drr]_{\pi_2} \ar[r]^{1_{\mathbb{Z} \times \mathbb{Z}}} & \mathbb{Z} \times \mathbb{Z} \ar[r]^-q \ar[ur]^{\pi_1} \ar[dr]_{\pi_2} & Q \ar[u]_{\varphi_1} \ar[d]^{\varphi_2} \\
& & \mathbb{Z}. } \]
So, we would have $\langle \varphi_1, \varphi_2 \rangle q = 1_{\mathbb{Z} \times \mathbb{Z}}$, which would imply that $q$ is a split monomorphism. However, in general, a $\mathcal{Z}$-cokernel is an epimorphism, hence $q$ would be an isomorphism. Therefore, up to isomorphisms, we can assume $q=1_{\mathbb{Z} \times \mathbb{Z}}$. But, if $1_{\mathbb{Z} \times \mathbb{Z}}$ were the $\mathcal{Z}$-cokernel of $1_{\mathbb{Z} \times \mathbb{Z}}$, we would obtain a factorization of the form
\[ \xymatrix{ \mathbb{Z} \times \mathbb{Z} \ar[dr]_p \ar[rr]^{1_{\mathbb{Z} \times \mathbb{Z}} = \langle \pi_1, \pi_2 \rangle} & & \mathbb{Z} \times \mathbb{Z} \\
& Z, \ar[ur]_g } \]
where $Z$ is a quotient of $\mathbb{Z}$. Thanks to the existence of the arrow $g$, we observe that the characteristic of $Z$ is $0$, and so $Z=\mathbb{Z}$ (up to isomorphisms). Moreover, we get $g=\Delta$. To conclude, we observe $\langle p, p\rangle = \Delta p = \langle \pi_1 , \pi_2 \rangle$, which implies $p=\pi_1=\pi_2$, which is a contradiction. \\

Using the same example, we can exhibit a regular epimorphism which is not a $\mathcal{Z}$-cokernel, despite the fact that $\mathbf{Ring}$ is a protomodular category. Indeed, the first projection $\pi_1 \colon \mathbb{Z} \times \mathbb{Z} \to \mathbb{Z}$ is not a $\mathcal{Z}$-cokernel: if there was a morphism $f \colon C \to \mathbb{Z} \times \mathbb{Z}$ whose $\mathcal{Z}$-cokernel is $\pi_1$, then, observing that $\pi_2 f \in \mathcal{N}_{\mathcal{Z}}$ since the codomain of such morphism is in $\mathcal{Z}$, there would exist a unique morphism $\gamma \colon \mathbb{Z} \to \mathbb{Z}$ such that $\gamma \pi_1 = \pi_2$, but this is impossible. \\

Going back to $\mathcal{Z}$-kernels, we can observe that a classical property of pointed categories, namely that parallel arrows in a pullback diagram have isomorphic kernels, does not hold in our context. A simple counterexample, in $\mathbf{Ring}$, is the following:
\[ \xymatrix{ \mathbb{Z} \ar@{=}[r] \ar[d] \pullback & \mathbb{Z} \ar[d] \\
\mathbb{Z}/2 \ar@{=}[r] & \mathbb{Z}/2, } \]
where the vertical morphisms are the canonical projections. However, we can recover the same kind of result when the morphism on the right is inverted by $\mathsf{Z}$:

\begin{prop}
Consider a pullback
\[ \xymatrix{ A \ar[r]^f \ar[d]_h \pullback & B \ar[d]^l \\
C \ar[r]_g & D, } \]
where $\mathsf{Z}(l)$ is an isomorphism. Then, $\Zker(f)$ and $\Zker(g)$ are isomorphic.
\end{prop}

\begin{proof}
Consider the following commutative cube:
\begin{equation}\label{cube kernels}
\xymatrix{ K(f) \ar[dd]_{\exists ! \varphi} \ar[dr]^k \ar[rr]^{\chi} & & \mathsf{Z}(B) \ar[dr]^{\varepsilon_B} \ar[dd]^(.7){\mathsf{Z}(l)} & \\
& A \ar[dd]_(.3)h \ar[rr]^(.3)f & & B \ar[dd]^l \\
K(g) \ar[dr]_n \ar[rr]^(.3){\chi'} & & \mathsf{Z}(D) \ar[dr]^{\varepsilon_D} & \\
& C \ar[rr]_g & & D, }
\end{equation}
where $k=\Zker(f)$, $n=\Zker(g)$, and $\varphi$ is induced by the universal property of the pullback. By assumption, the top face, the front face and the bottom face are pullbacks. Then the back face is a pullback. Therefore, $\varphi$ is an isomorphism, as pullback of the isomorphism $\mathsf{Z}(l)$.
\end{proof}

If, moreover, $\mathcal{C}$ is a protomodular category, a partial converse holds:

\begin{prop}
Consider a commutative diagram
\[ \xymatrix{ A \ar@{->>}[r]^f \ar[d]_h & B \ar[d]^l \\
C \ar[r]_g & D, } \]
where $f$ is a regular epimorphism, $\mathsf{Z}(l)$ is an isomorphism, and $\Zker(g)$ is isomorphic to $\Zker(f)$. Then the square is a pullback.
\end{prop}

\begin{proof}
Consider the commutative cube \eqref{cube kernels}, where again $k=\Zker(f)$, $n=\Zker(g)$, and $\varphi$
is induced by the universal property of the pullback. By assumption, the top face, the bottom face, and the back face (since $\varphi$ and $\mathsf{Z}(l)$ are isomorphisms) are pullbacks. Therefore, we obtain the commutative diagram
\[ \xymatrix{ K(f) \ar@{ m>->}[r]^k \ar[d]_{\chi} \ar@{}[dr]|{(1)} & A \ar[r]^h \ar@{->>}[d]_f \ar@{}[dr]|{(2)} & C \ar[d]^g \\
\mathsf{Z}(B) \ar@{ m>->}[r]_{\varepsilon_B} & B \ar[r]_l & D } \]
where both (1) and (1)+(2) are pullbacks and $f$ is a regular epimorphism. By protomodularity, we conclude that (2) is a pullback, too.
\end{proof}

It is well known that, in a pointed, regular, protomodular category, a morphism $f$ is a monomorphism if and only if its kernel is $0$. In our non-pointed case, however, the analogue of this result does not always hold. For instance, consider the canonical projection $p_n \colon \mathbb{Z} \to \mathbb{Z}/n$ in $\mathbf{Ring}$. It is immediate to observe that $\Zker(p_n)=1_{\mathbb{Z}}$, yet $p_n$ is not a monomorphism. However, if we consider a morphism $f \colon A \to B$ inverted by the functor $\mathsf{Z}$, then it is true that $f$ is a monomorphism if and only if $\Zker(f)=\varepsilon_A$. In fact, if the diagram
\[ \xymatrix{ K(f) \ \ar@{->>}[r]^{\chi} \ar@{ m>->}[d]_k  & \mathsf{Z}(B) \ar@{ m>->}[d]^{\varepsilon_B} \\
A \ar@{ m>->}[r]_f & B } \]
is a pullback, then $\chi$ is an isomorphism ($\chi$ is a strong epimorphism being an arrow with trivial codomain, and it is a monomorphism as pullback of a monomorphism), and so $k$ is isomorphic to $\varepsilon_A$. Conversely, if $\varepsilon_A=\Zker(f)$ the diagram
\[ \xymatrix{ \mathsf{Z}(A) \ \ar[r]^{\mathsf{Z}(f)}_{\sim} \ar@{ m>->}[d]_{\varepsilon_A}  & \mathsf{Z}(B) \ar@{ m>->}[d]^{\varepsilon_B} \\
A \ar[r]_f & B } \]
is a pullback, and so, by protomodularity, $f$ is a monomorphism since its pullback $\mathsf{Z}(f)$ is a monomorphism.

\section{Homological lemmas}

Now we are ready to prove our non-pointed version of the classical homological lemmas (the proofs will be adaptations of the ones known for homological categories, that can be found e.g.\ in \cite{Bbbook}). To do that, we first have to define short exact sequences in our context. Let $\mathcal{C}$ be a regular protomodular category equipped with a full subcategory $\mathcal{Z}$ satisfying the equivalent conditions of Proposition \ref{contesto generale}.

\begin{defi}
In the category $\mathcal{C}$, the sequence
\[ \xymatrix{ A \ar[r]^k & B \ar[r]^f & C } \]
of objects and morphisms is said to be \emph{a short $\mathcal{Z}$-exact sequence} if $f$ is a regular epimorphism and $k=\Zker(f)$.
\end{defi}

As we observed in the previous section, this request does not imply that $f$ is the $\mathcal{Z}$-cokernel of $k$. This is the main difference between our context on one side, and, on the other side, the classical homological one, the quasi-pointed one \cite{Bourn 3x3}, and the non-pointed ones considered in \cite{Grandis, GrandisJanMarki, GJR 3x3 lemma}. \\

Let us begin with the short five lemma:

\begin{prop} \label{prop short five lemma}
Consider a commutative diagram
\[ \xymatrix{ K \ar[r]^k \ar[d]_u & A \ar[r]^f \ar[d]^a & B \ar[d]^b \\
K' \ar[r]_{k'} & A' \ar[r]_{f'} & B' } \]
in which the rows are short $\mathcal{Z}$-exact sequences. If $u$ and $b$ are isomorphisms, then $a$ is an isomorphism, too.
\end{prop}

\begin{proof}
Consider the following commutative diagram:
\[ \xymatrix{ K \ar[dd]_u \ar[dr]^k \ar[rr] & & \mathsf{Z}(B) \ar[dr] \ar[dd]^(.7){\mathsf{Z}(b)} & \\
& A \ar[dd]_(.3)a \ar[rr]^(.3)f & & B \ar[dd]^b \\
K' \ar[dr]_{k'} \ar[rr] & & \mathsf{Z}(B') \ar[dr] & \\
& A' \ar[rr]_{f'} & & B', } \]
where, by assumption, the top face, the bottom face, and the back face are pullbacks. Then, the diagram defined by the top face and the front face is a pullback. Therefore, we obtain the commutative diagram
\[ \xymatrix{ K \ar[r]^k \ar[d] \ar@{}[dr]|{(1)} & A \ar[r]^a \ar@{->>}[d]^f \ar@{}[dr]|{(2)} & A' \ar[d]^{f'} \\
\mathsf{Z}(B) \ar[r]_{\varepsilon_B} & B \ar[r]_b & B',} \]
where both (1) and (1)+(2) are pullbacks and $f$ is a regular epimorphism. By protomodularity, we conclude that (2) is a pullback, too. We deduce that $a$ is an isomorphism, as pullback of the isomorphism $b$.
\end{proof}

In order to split the lemma in the versions for monomorphisms and for regular epimorphisms, we need the following well-known lemma, which holds in every regular category:

\begin{lem} \label{lemma pb reg epi}
If, in the following diagram, \emph{(1)} and \emph{(2)} are commutative squares, such that \emph{(1)} and \emph{(1)+(2)} are pullbacks and $f'$ is a regular epimorphism, then \emph{(2)} is a pullback:
\[ \xymatrix{ A \ar[r]^f \ar[d]_a \ar@{}[dr]|{(1)} & B \ar[r]^g \ar[d]^b \ar@{}[dr]|{(2)} & C \ar[d]^c \\
A' \ar@{->>}[r]_{f'} & B' \ar[r]_{g'} & C'.} \]
\end{lem}

\begin{lem} \label{lemma pb iff mono}
Consider a commutative diagram
\[ \xymatrix{ K \ar[r]^k \ar[d]_u \ar@{}[dr]|{(1)} & A \ar[r]^f \ar[d]^a & B \ar[d]^b \\
K' \ar[r]_{k'} & A' \ar[r]_{f'} & B'  } \]
in which the top row is $\mathcal{Z}$-exact, $k' = \Zker(f')$, and $\mathsf{Z}(B) \cong \mathsf{Z}(B')$. Then the square (1) is a pullback if and only if $b$ is a monomorphism.
\end{lem}

\begin{proof}
We consider the commutative cube
\[ \xymatrix{ K \ar[dd]_u \ar[dr] \ar[rr]^k & & A \ar[dr]^f \ar[dd]^(.7)a & \\
& \mathsf{Z}(B) \ar[dd]^(.3){\simeq} \ar[rr] & & B \ar[dd]^b \\
K' \ar[dr] \ar[rr]^(.3){k'} & & A' \ar[dr]^{f'} & \\
& \mathsf{Z}(B') \ar[rr] & & B', } \]
where the top face and the bottom face are pullbacks. If (1) is a pullback, then the diagram formed by the top face and the front face is a pullback, too. Therefore, in the commutative diagram
\[ \xymatrix{ K \ar[r] \ar[d]_k \ar@{}[dr]|{(3)} & \mathsf{Z}(B) \ar[r]^{\simeq} \ar[d] \ar@{}[dr]|{(4)} & \mathsf{Z}(B') \ar[d] \\
A \ar@{->>}[r]_{f} & B \ar[r]_b & B'} \]
we have that (3)+(4) is a pullback, (3) is a pullback and $f$ is a regular epimorphism. Then, thanks to Lemma \ref{lemma pb reg epi}, we obtain that (4) is a pullback, as well. Since in a protomodular category pullbacks reflect monomorphisms, we can conclude that $b$ is a monomorphism.\\
Conversely, if we assume that $b$ is a monomorphism, we can deduce that the front face in the cube above is a pullback. Therefore, the diagram defined by the back face and  the bottom face forms a pullback. Since the bottom face is already a pullback, we can conclude that the back face is a pullback, too. In other words, (1) is a pullback.
\end{proof}

We observe that the assumption $\mathsf{Z}(B) \cong \mathsf{Z}(B')$ in the previous lemma is essential. This can be seen considering the commutative diagram below, in the category of unitary rings:
\[ \xymatrix{ \mathbb{Z} \ar@{}[dr]|{(1)} \ar[d] \ar@{=}[r] & \mathbb{Z} \ar[d] \ar@{=}[r] & \mathbb{Z} \ar[d]^b \\
1 \ar@{=}[r] & 1 \ar@{=}[r] & 1. } \]
In fact, (1) is a pullback but $b$ is not a monomorphism.

\begin{prop} \label{short five lemma for reg epis}
Given the commutative diagram below, with $\mathcal{Z}$-exact rows:
\begin{equation} \label{diagram sfl}
\xymatrix{ K \ar[r]^k \ar[d]_u & A \ar[r]^f \ar[d]^a & B \ar[d]^b \\
K' \ar[r]_{k'} & A' \ar[r]_{f'} & B', }
\end{equation}
\begin{itemize}
\item[(a)] if $u$ and $b$ are regular epimorphisms, then $a$ is a regular epimorphism, too;

\item[(b)] if $u$ and $b$ are monomorphisms, then $a$ is a monomorphism, too.
\end{itemize}
\end{prop}

\begin{proof}
\begin{itemize}
\item[(a)] Consider the factorization $a=ip$, where $p$ is a regular epimorphism and $i$ is a monomorphism, and the following commutative diagram:
\[ \xymatrix{ K \ar[r]^k \ar@{->>}[d]_u & A \ar[r]^f \ar@{->>}[d]^p & B \ar@{->>}[d]^b \\
K' \ar[r]^-j \ar@{=}[d] & \text{Im}(a) \ar[r]^-{f' i} \ar@{ m>->}[d]^i & B' \ar@{=}[d] \\
K' \ar[r]_{k'} & A' \ar@{->>}[r]_{f'} & B', } \]
where $j$ is determined by the fact that $u$ is a strong epimorphism and $i$ is a monomorphism. Then $j$ is the $\mathcal{Z}$-kernel of $f'i$: in fact, in the diagram below the two squares are pullbacks:
\[ \xymatrix{ K' \ar@{=}[r] \ar[d]_j & K' \ar[r] \ar[d]^{k'} & \mathsf{Z}(B') \ar[d] \\
\text{Im}(a) \ar@{ m>->}[r]_-i & A' \ar[r]_{f'} & B', } \]
and thus the whole rectangle is a pullback, too. Furthermore, we observe that $f' i$ is a regular epimorphism, since $f' i p = bf$ is; also $f'$ is a regular epimorphism, and so, applying Proposition \ref{prop short five lemma} to the bottom rectangle of the first diagram in this proof, we can deduce that $i$ is an isomorphism. Hence, $a$ is a regular epimorphism.

\item[(b)] Being $b$ a monomorphism, we have that $\mathsf{Z}(B) \cong \mathsf{Z}(B')$. Then we can apply Lemma \ref{lemma pb iff mono} to Diagram \eqref{diagram sfl}, obtaining that the left square in \eqref{diagram sfl} is a pullback. Since in a protomodular category pullbacks reflect monomorphisms, we conclude that $a$ is a monomorphism, too.
\end{itemize}
\end{proof}

Now we study the nine lemma, starting with an auxiliary result:

\begin{lem} \label{lemma reg epi iff reg epi}
Given a commutative diagram
\[ \xymatrix{ & A \ar[r]^f \ar[d]_a & B \ar[d]^b  \\
K' \ar[d]_u \ar[r]^{k'} & A' \ar[d]_{a'} \ar[r]^{f'} & B' \ar[d]^{b'} \\
K'' \ar[r]_{k''} & A'' \ar[r]_{f''} & B'' } \]
with $\mathcal{Z}$-exact rows and columns, if $\mathsf{Z}(B') \cong \mathsf{Z}(B'') \cong \mathsf{Z}(A'')$ and $u$ is a regular epimorphism, then $f$ is, too.
\end{lem}

\begin{proof}
Consider the diagram
\begin{equation}\label{diagram 1}
\xymatrix{ K' \ar[d]_u \ar[r]^{k'} & A' \ar[d]_{\langle a', f' \rangle} \ar[r]^{f'} & B' \ar@{=}[d] \\
K'' \ar[r]_-{\langle k'', \varepsilon_{B'} \alpha \rangle} & A'' \times_{B''} B' \ar[r]_-{\pi_{B'}} & B' }
\end{equation}
where $\alpha$ is the upper arrow in the following pullback square:
\[ \xymatrix{ K'' \ar[r]^{\alpha} \ar[d]_{k''} \pullback & \mathsf{Z}(B'') \ar[d]^{\varepsilon_{B''}} \\
A'' \ar[r]_{f''} & B''. } \]
In order to prove the commutativity of \eqref{diagram 1}, we have to show that $f'k'= \varepsilon_{B'} \alpha u$. We know that $f'k'= \varepsilon_{B'} \chi$ (where $\chi \colon K' \rightarrow \mathsf{Z}(B')$). We observe that $b'f'k'=f''k''u= \varepsilon_{B''} \alpha u$; moreover, we notice that $b'f'k'=b' \varepsilon_{B'} \chi= \varepsilon_{B''} \chi$ (with a slight abuse of notation, we are identifying $\mathsf{Z}(B')$ and $\mathsf{Z}(B'')$). Since $\varepsilon_{B''}$ is a monomorphism, we deduce $\alpha u = \chi$, and so $\varepsilon_{B'} \alpha u= \varepsilon_{B'} \chi=f'k'$. We prove that $\langle k'', \varepsilon_{B'} \alpha \rangle$ is the $\mathcal{Z}$-kernel of $\pi_{B'}$. To see this, we notice that in the diagram
\[ \xymatrix{ K'' \ar@/_3pc/[dd]_{k''} \ar[r]^{\alpha} \ar[d]^{\langle k'', \varepsilon_{B'} \alpha \rangle} & \mathsf{Z}(B') \ar@/^3pc/[dd]^{\varepsilon_{B''}} \ar[d]^{\varepsilon_{B'}} \\
A'' \times_{B''} B' \ar[d]^{\pi_{A''}} \ar[r]_-{\pi_{B'}} & B' \ar[d]^{b'} \\
A'' \ar[r]_{f''} & B'' } \]
the rectangle and the bottom square are pullbacks, hence the top square is a pullback, too. If $u$ is a regular epimorphism, then by Proposition \ref{short five lemma for reg epis} we can deduce that $\langle a', f' \rangle$ is a regular epimorphism. We can then construct the commutative diagram:
\begin{equation}\label{diagram 2}
\xymatrix{ A \ar[d]_f \ar[r]^a & A' \ar[d]_{\langle a', f' \rangle} \ar[r]^{a'} & A'' \ar@{=}[d] \\
B \ar[r]_-{\langle \varepsilon_{A''} \beta, b \rangle} & A'' \times_{B''} B' \ar[r]_-{\pi_{A''}} & A'' }
\end{equation}
where $\beta$ is the upper arrow in the following pullback square (again, with a slight abuse of notation, we are identifying $\mathsf{Z}(A'')$ and $\mathsf{Z}(B'')$):
\[ \xymatrix{ B \ar[r]^{\beta} \ar[d]_b \pullback & \mathsf{Z}(B'') \ar[d]^{\varepsilon_{B''}} \\
B' \ar[r]_{b'} & B''. } \]
The commutativity of \eqref{diagram 2} can be proved using a similar argument as the one used for \eqref{diagram 1}. We observe that $\langle \varepsilon_{A''} \beta , b \rangle$ is the $\mathcal{Z}$-kernel of $\pi_{A''}$. In fact, we have the commutative rectangle
\[ \xymatrix{ B  \ar@{}[dr]|{(1)} \ar[r]^-{\langle \varepsilon_{A''} \beta, b \rangle} \ar[d]_{\beta} & A'' \times_{B''} B' \ar@{}[dr]|{(2)} \ar@{->>}[d]_{\pi_{A''}} \ar[r]^-{\pi_{B'}} & B' \ar[d]^{b'} \\
\mathsf{Z}(A'') \ar[r]_{\varepsilon_{A''}} & A'' \ar[r]_{f''} & B'' } \]
where (2) and (1)+(2) are pullbacks, and so (1) is a pullback, too. Hence, applying Lemma \ref{lemma pb iff mono} to \eqref{diagram 2}, we obtain that the left square of \eqref{diagram 2} is a pullback, since $1_{A''}$ is a monomorphism. So, by regularity, we conclude that $f$ is a regular epimorphism.
\end{proof}

\begin{teo} \label{teo nine lemma}
Consider the following commutative diagram:
\begin{equation} \label{3x3 diagram}
\xymatrix{ K \ar[d]_k \ar[r]^u & K' \ar[d]^{k'} \ar[r]^{u'} & K'' \ar[d]^{k''} \\
A \ar[d]_f \ar[r]^a & A' \ar[d]^{f'} \ar[r]^{a'} & A'' \ar[d]^{f''} \\
B \ar[r]_b & B' \ar[r]_{b'} & B''  }
\end{equation}
with $\mathcal{Z}$-exact columns and such that $\mathsf{Z}(B') \cong \mathsf{Z}(B'') \cong \mathsf{Z}(A'')$. Then:
\begin{itemize}
\item[(a)] if $\mathsf{Z}(B) \cong \mathsf{Z}(B')$ and the first two rows are $\mathcal{Z}$-exact, the third row also is;

\item[(b)] if the last two rows are $\mathcal{Z}$-exact, the first row also is;

\item[(c)] if $a' a$ factors through $\mathcal{Z}$ and the first and the last row are $\mathcal{Z}$-exact, the middle row also is.
\end{itemize}
\end{teo}

\begin{proof}
\begin{itemize}
\item[(a)] The fact that $b'$ is a regular epimorphism follows from the observation that $b'f' = f''a'$ is. Being $k''$ a monomorphism, we get $\mathsf{Z}(K'') \cong \mathsf{Z}(A'')$ and we can apply Lemma \ref{lemma pb iff mono} to the upper part of Diagram \eqref{3x3 diagram} to conclude that the upper left square is a pullback. The assumption that $\mathsf{Z}(B) \cong \mathsf{Z}(B')$ allows to apply Lemma \ref{lemma pb iff mono} again to the left part of Diagram \eqref{3x3 diagram}, giving that $b$ is a monomorphism. Let $h \colon K(b') \to B'$ be the $\mathcal{Z}$-kernel of $b'$. Consider the commutative diagram
\[ \xymatrix{ & & & B \ar[dl]^l \ar@/^/[ddl]^b \\
K \ar[d]_u \ar[r]^k & A \ar[d]_a \ar[r]^g \ar@/^/[urr]^f & K(b') \ar[d]^h & \\
K' \ar[r]_{k'} & A' \ar[r]_{f'} & B' \ar[d]^{b'} & \\
& & B'', & } \]
where $g$ is determined by the fact that $b'f'a=f''a'a$ factors through $\mathcal{Z}$. Moreover, observing that $b'bf=f''a'a$ factors through $\mathcal{Z}$ and applying Lemma \ref{cancellation in NZ with strong epis}, we obtain that $b'b$ factors through $\mathcal{Z}$. Hence, there exists a unique arrow $l$ such that $hl=b$. Recalling that $hlf=bf=f'a=hg$ and that $h$ is a monomorphism, we show that $lf = g$. Thanks to Lemma \ref{lemma reg epi iff reg epi} we observe that $g$ is a regular epimorphism, so $l$ is a regular epimorphism, too. Furthermore, since $b$ is a monomorphism, we get that $l$ is a monomorphism. Thus, $l$ is an isomorphism.

\item[(b)] By assumption $b \colon B \to B'$ is a monomorphism, so $\mathsf{Z}(B) \cong \mathsf{Z}(B')$ and, by Lemma \ref{lemma pb iff mono}, the upper left square in Diagram \eqref{3x3 diagram} is a pullback. Consider then the commutative cube
\[ \xymatrix{ K \ar[dd]_u \ar[dr]^k \ar[rr]^{\kappa} & & \mathsf{Z}(K'') \ar[dr]^{\simeq} \ar[dd] & \\
& A \ar[dd]_(.3)a \ar[rr]^(.3){\alpha} & & \mathsf{Z}(A'') \ar[dd] \\
K' \ar[dr]_{k'} \ar[rr]^(.7){u'} & & K'' \ar[dr]^{k''} & \\
& A' \ar[rr]_{a'} & & A'', } \]
where $\mathsf{Z}(K'') \cong \mathsf{Z}(A'')$ because $k''$ is a monomorphism, the existence of the arrow $\alpha$ is guaranteed by the fact that $a' a$ factors through $\mathsf{Z}(A'')$, and similarly $\kappa$ is given by the fact that $u' u$ factors through $\mathsf{Z}(K'')$. The left, right and front faces of the cube are pullbacks, because $a$ is the $\mathcal{Z}$-kernel of $a'$. Hence the back face is a pullback, proving that $u$ is the $\mathcal{Z}$-kernel of $u'$. Moreover, $u'$ is a regular epimorphism thanks to Lemma \ref{lemma reg epi iff reg epi}, since $\mathsf{Z}(B') \cong \mathsf{Z}(B'') \cong \mathsf{Z}(A'')$.

\item[(c)] Applying Proposition \ref{short five lemma for reg epis} to the right part of Diagram \eqref{3x3 diagram} we get that $a'$ is a regular epimorphism. Let $n' \colon K(a') \to A'$ be its $\mathcal{Z}$-kernel. Consider the following diagram:
\[ \xymatrix{ K \ar@{=}[d] \ar[r]^k & A \ar[d]^{\varphi} \ar[r]^f & B \ar@{=}[d] \\
K \ar[r]_{\varphi k} & K(a') \ar[d]^{n'} \ar[r]_j & B \ar[d]^b \\
& A' \ar[d]^{a'} \ar[r]_{f'} & B' \ar[d]^{b'} \\
& A'' \ar[r]_{f''} & B'', } \]
where $\varphi$ is induced by the fact that $a' a$ factors through $\mathcal{Z}$ and, similarly, $j$ is induced by the fact that $b' f' n' = f'' a' n'$ factors through $\mathcal{Z}$ and $b$ is the $\mathcal{Z}$-kernel of $b'$. The upper right square commutes, since $bj\varphi = f' n' \varphi = f' a = bf$ and $b$ is a monomorphism. This implies, in particular, that $j$ is a regular epimorphism. Moreover, we can conclude that $\varphi k$ is the $\mathcal{Z}$-kernel of $j$ by applying point (b) of the present theorem to the diagram
\[ \xymatrix{ K \ar[d]_u \ar[r]^{\varphi k} & K(a') \ar[d]^{n'} \ar[r]^{j} & B \ar[d]^b \\
K' \ar[d]_{u'} \ar[r]^{k'} & A' \ar[d]^{a'} \ar[r]^{f'} & B' \ar[d]^{b'} \\
K'' \ar[r]_{k''} & A'' \ar[r]_{f''} & B''. } \]
Then we can apply Proposition \ref{prop short five lemma} to the diagram
\[ \xymatrix{ K \ar@{=}[d] \ar[r]^k & A \ar[d]^{\varphi} \ar[r]^f & B \ar@{=}[d] \\
K \ar[r]_{\varphi k} & K(a') \ar[r]_{j} & B } \]
to conclude that $\varphi$ is an isomorphism.
\end{itemize}
\end{proof}

In the case of pointed or quasi-pointed regular protomodular categories, where the subcategory $\mathcal{Z}$ is reduced to the initial object, every regular epimorphism is the $\mathcal{Z}$-cokernel of its $\mathcal{Z}$-kernel (as observed in \cite{Bourn 3x3}). So the notion of short exact sequence considered in \cite{Bourn 3x3} coincide with ours, and our previous theorem reduces to the nine lemma proved there. However, our version applies to several examples that are not covered by the results in \cite{Bourn 3x3} nor by the star-regular context with enough trivial objects considered in \cite{GJR 3x3 lemma}. Among these examples there are the duals of elementary toposes and the protomodular varieties with more than one constant. Indeed, in this context there are not enough trivial objects in the sense of \cite{GJR 3x3 lemma}, since it is not true that the subcategory $\mathcal{Z}$ is closed under squares. For example, in $\mathbf{Ring}$, $\mathbb{Z} \times \mathbb{Z}$ is not in $\mathcal{Z}$.

\section*{Acknowledgements}
Both authors are members of the Gruppo Nazionale per le Strutture Algebriche, Geometriche e le loro Applicazioni (GNSAGA) dell'Istituto Nazionale di Alta Matematica ``Francesco Severi''.

This work was supported by Shota Rustaveli National Science Foundation of Georgia (SRNSFG), through grant FR-24-9660.


\begin{thebibliography}{99}

\bibitem{Bbbook} F. Borceux, D. Bourn, \emph{Mal'cev, protomodular, homological and semi-abelian
categories}, Mathematics and its applications, vol. 566 (2004), Kluwer.

\bibitem{BorceuxClementino protomod top alg} F. Borceux, M.M. Clementino, \emph{Topological protomodular algebras}, Topology and its Appl. 153 (2006), 3085-3100.

\bibitem{Bourn protomod} D. Bourn, \emph{Normalization equivalence, kernel equivalence and affine
categories}, Category theory (Como, 1990), Lecture Notes in Math., vol. 1488,
Springer, Berlin, 1991, 43–62.

\bibitem{Bourn 3x3} D. Bourn, \emph{$3 \times 3$ lemma and protomodularity}, J. Algebra 236 (2001), 778–795.

\bibitem{Bourn denormalized 3x3} D. Bourn, \emph{The denormalized $3 \times 3$ lemma}, J. Pure Appl. Algebra 177 (2003), 113-129.

\bibitem{CKP Goursat cats} A. Carboni, G. M. Kelly, M. C. Pedicchio, \emph{Some remarks on Maltsev and
Goursat categories}, Appl. Cat. Struct. 1 (1993), 385–421.

\bibitem{CLP Malt'sev cats} A. Carboni, J. Lambek, M. C. Pedicchio, \emph{Diagram chasing in Malcev categories},
J. Pure Appl. Alg. 69 (1991), 271–284.

\bibitem{Ehresmann} C. Ehresmann, \emph{Sur une notion g\'{e}n\'{e}rale de cohomologie}, C. R. Acad. Sci. Paris
259 (1964), 2050–2053.

\bibitem{Ehresmann2} C. Ehresmann, \emph{Cohomologie \`{a} valeurs dans une cat\'{e}gorie domin\'{e}e}, Extraits du Colloque de Topologie, Bruxelles, 1964, in: Oeuvres Compl\`{e}tes et Comment\'{e}es, Partie III-2 (1980), 531-590.

\bibitem{GJR 3x3 lemma} M. Gran, Z. Janelidze, D. Rodelo, \emph{$3 \times 3$ lemma for star-exact sequences}, Homology Homotopy Appl. 14 n.2 (2012), 1-22.

\bibitem{GJU star-regular} M. Gran, Z. Janelidze, A. Ursini, \emph{A good theory of ideals in regular multi-pointed categories}, J. Pure Appl. Algebra 216 (2012), 1905-1919.

\bibitem{Grandis} M. Grandis, \emph{On the categorical foundations of homological and homotopical
algebra}, Cah. Top. G\'{e}om. Diff. Cat\'{e}g. 33 (1992), 135–175.

\bibitem{GrandisJanMarki} M. Grandis, G. Janelidze, L. M\'{a}rki, \emph{Non-pointed exactness, radicals, closure operators}, J. Aust. Math. Soc. 94 (2013), 348-361.

\bibitem{Janelidze ideally exact} G. Janelidze, \emph{Ideally exact categories}, Theory Appl. Categ. 41 n.11 (2024), 414-425.

\bibitem{Lack 3x3} S. Lack, \emph{The 3-by-3 lemma for regular Goursat categories}, Homology Homotopy Appl. 6 (2004), 1–3.

\end{thebibliography}
\end{document}